\documentclass[11pt]{article}
\usepackage[margin=1in]{geometry}

\usepackage[all]{xy}
\usepackage{tikz}
\usepackage{enumerate}

	\usepackage{authblk} 
    \usepackage{amssymb}
    \usepackage{amsmath}
    \usepackage{amsthm}
\IfFileExists{euscript.sty}{\usepackage[mathcal]{euscript}}{} 
    \usepackage{url}
    \usepackage[colorlinks=true, pdfstartview=FitV, linkcolor=black,citecolor=black, urlcolor=black]{hyperref}
    \usepackage[titletoc]{appendix} 
    \hypersetup{hypertexnames=false} 

          \newcommand{\nc}{\newcommand}
          \nc{\DMO}{\DeclareMathOperator}	
          \DMO{\op}{op}
          \DMO{\ob}{Ob}
          \DMO{\id}{id}
          \DMO{\Map}{Map}
          \DMO{\aut}{Aut}
          \DMO{\cone}{Cone}
          \DMO{\colim}{colim}
          \DMO{\stable}{Ex}
          \DMO{\hocolim}{hocolim}	
          \nc{\xto}{\xrightarrow}
          \nc{\xra}{\xto}
          \nc{\tensor}{\otimes}
          \nc{\Tw}{\mathsf{Tw}}
          \nc{\into}{\hookrightarrow}
          \nc{\sets}{\mathsf{Sets}}
          \nc{\sing}{\mathsf{Sing}}
          \nc{\sset}{\mathsf{sSet}}
          \nc{\chain}{\mathsf{Chain}}
          \nc{\fun}{\mathsf{Fun}}
			\nc{\Etarget}{E}

          \nc{\Cat}{\mathsf{Cat}}
          \nc{\Mod}{\mathsf{Mod}}
          \nc{\inftyCat}{\mathcal{C}\!\operatorname{at}_\infty}
          \nc{\inftyGpd}{\mathcal{G}\!\operatorname{pd}_\infty}
          \nc{\Ainftycat}{\mathcal{C}\!\operatorname{at}_{A_\infty}}
          \nc{\dgcat}{\mathcal{C}\!\operatorname{at}_{dg}}

          \nc{\Ainftycatt}{A_\infty Cat}
          \nc{\dgcatt}{dg Cat}
          \nc{\Kan}{\mathcal{K}\!\operatorname{an}}
          \nc{\eqn}{\begin{equation}}
          \nc{\eqnn}{\begin{equation}\nonumber}
          \nc{\eqnd}{\end{equation}}
          \nc{\eqnnd}{\end{equation}}
          \nc{\enum}{\begin{enumerate}}
          \nc{\enumd}{\end{enumerate}}
          
          \def\cA{\mathcal A}
          \def\cB{\mathcal B}
          \def\cC{\mathcal C}
          \def\cD{\mathcal D}
          \def\cE{\mathcal E}

          \def\cL{\mathcal L}
          \def\cQ{\mathcal Q}
          \def\cX{\mathcal X}
          \def\cY{\mathcal Y}
          \def\ZZ{\mathbb Z}

          \def\sD{\mathsf D}
          \def\sq{\mathsf q}
			\def\kk{\mathbf k}
          \theoremstyle{definition}
          \newtheorem{theorem}{Theorem}[section]
          \newtheorem{proposition}[theorem]{Proposition}
          \newtheorem{prop}[theorem]{Proposition}
          \newtheorem{lemma}[theorem]{Lemma}

          \newtheorem{definition}[theorem]{Definition}
          \newtheorem{defn}[theorem]{Definition}
          \newtheorem{notation}[theorem]{Notation}

          \newtheorem{remark}[theorem]{Remark}

\nc{\hiro}{\textcolor{blue}}

\title{$\infty$-categorical universal properties of quotients and localizations of $A_\infty$-categories}

\author[$\dagger$]{Yong-Geun Oh}
\author[$\star$]{Hiro Lee Tanaka}

\affil[$\dagger$]{Center for Geometry and Physics (IBS), Pohang, Korea \&
Department of Mathematics, POSTECH, Pohang Korea.}
\affil[$\star$]{Department of Mathematics, Texas State University}

\begin{document}
\maketitle

\begin{abstract}
We show that certain hands-on $A_\infty$-categorical constructions satisfy desirable universal properties in the infinity-category of $A_\infty$-categories. For sufficiently cofibrant $A_\infty$-categories, two models for quotients of $A_\infty$-categories (as constructed by Lyubashenko-Manzyuk and Lyubashenko-Ovisienko), and a model for localizations (as used by Ganatra-Pardon-Shende), satisfy the relevant universal properties. We apply the results here in a companion work to prove a Liouville version of a conjecture of Teleman from the 2014 ICM. 
\end{abstract}

\tableofcontents

\section{Introduction}

Throughout, we fix a base commutative ring $\kk$. All $A_\infty$-categories and functors are assumed $\kk$-linear and (not necessarily strictly) unital; in particular, in the category of $A_\infty$-categories, we take objects to be unital, and all morphisms to be unital functors. 

The aim of this paper is to show that $A_\infty$-categories that have been called quotients and localizations in previous works merit the labels: They indeed satisfy, $\infty$-categorically, the expected universal properties. Such a result is available to us because we now have a satisfactory set-up for the $\infty$-category of $A_\infty$-categories. (A previous version of this work gave a (correct) definition of the $\infty$-category of $A_\infty$-categories, but had a gap in the proof that this $\infty$-category had certain desirable properties. These gaps have been filled thanks to~\cite{cos-over-rings,pascaleff,tanaka-Aoo-units}.)

Let us say what we mean. In the worlds of sets, of spaces, of groups, et cetera, a quotient of an object $A$ by another object $B$ satisfies the following universal property:
	\eqn\label{eqn. universal property quotient}
	\xymatrix{
	B \ar[r] \ar[d] & A \ar[ddrr] \ar[d] \\
	\ast \ar[r] \ar[drrr] & A/B \ar@{-->}[drr]^-{\exists !} \\
	&&&C.
	}
	\eqnd
In words, given any map $A \to C$ such that the composite $B \to A \to C$ factors through $\ast$, there exists a {\em unique} map from $A/B$ to $C$ making the above diagram commute. 

Regardless of what model one makes for an object deserving to be called $A/B$, one must verify the universal property for the model to be of formal use.

Likewise, 
given an $A_\infty$-category $\cA$, a full subcategory $\cB \subset \cA$, 
one can ask for the universal $A_\infty$-category $\cA/\cB$ satisfying a homotopical version of the universal property~\eqref{eqn. universal property quotient}. 

In a world with homotopies and higher homotopies, one never expects a unique functor $\cA/\cB \to \cC$ making a diagram commute; rather, one expects a {\em contractible} collection of functors equipped with data making a diagram commute up to higher homotopies. 

The reader may appreciate the burden of verifying such a homotopical unicity; there is much coherence to carry around. On the flip side, once one has proven that a universal property is satisfied, one has a powerful tool at one's disposal.  Our results here show that concrete models satisfy these powerful universal properties.

\begin{definition}
Following Notation~1.4 of~\cite{tanaka-Aoo-units},
	\eqnn
	\Ainftycat
	\eqnd
denotes the $\infty$-category of $A_\infty$-categories. We define this $\infty$-category as the $\infty$-categorical localization, along quasi-equivalences, of the usual category of $A_\infty$-categories and their functors. 
\end{definition}

\begin{defn}
\label{defn. cofibrant}
We will say that an $A_\infty$-category is {\em cofibrant} if all morphism complexes are homotopically projective. 
\end{defn}

See Remark~1.6, Remark~1.7, and Definition~3.7  of~\cite{tanaka-Aoo-units} for discussion regarding homotopically projective complexes. We use the term cofibrant despite the fact that there is no model structure on the category of $A_\infty$-categories.

Our two main results are as follows. 
Theorem~\ref{theorem. quotient models} is a generalization of Tabuada's paper on Drinfeld's dg quotient~\cite{tabuada-on-dg-quotient}. 

\begin{theorem}[Theorem~\ref{theorem. A oo quotient}] \label{theorem. quotient models}
Let $\cA$ be a cofibrant $A_\infty$-category and fix a full subcategory $\cB \subset \cA$. We let $\sq(\cA|\cB)$ denote the model for the quotient $A_\infty$-category from Lyubashenko-Manzyuk~\cite{lyubashenko-manzyuk}, and we let $\sD(\cA|\cB)$ denote the model from Lyubashenko-Ovsienko~\cite{lyubashenko-ovsienko}. 
Both models satisfy the universal property of quotients in $\Ainftycat$.
\end{theorem}

\begin{theorem}[Theorem~\ref{thm. localization of A oo cats}.]
Let $\cA$ be a cofibrant $A_\infty$-category and fix a collection of (cohomology classes of) morphisms $W \subset H^0\cA$. Then the localization $\cA[W^{-1}]$ as utilized in~\cite{gps-covariant} satisfies the universal property of localizations.
\end{theorem}

\begin{remark}[Necessity of cofibrancy]
The above cofibrancy requirements are necessary. This is because quotients and localizations have universal properties expressed by maps out of them, while mapping spaces out of an object are only well-behaved when the object is cofibrant. To illustrate this point: One would like the morphism space $\hom_{\Ainftycat}(\cA,\cE)$ between two $A_\infty$-categories to have the property that its $\pi_0$ is in bijection with the set of functors $\cA \to \cE$ up to homotopy (i.e., up to natural equivalence). At the same time, because any quasi-equivalence $f: \cA' \to \cA$ is homotopy invertible in $\Ainftycat$, $f^*: \hom_{\Ainftycat}(\cA,\cE) \to \hom_{\Ainftycat}(\cA',\cE)$ is a homotopy equivalence -- and in particular, a bijection on homotopy classes of functors. When $\cA$ and $\cA'$ are not both cofibrant, however, $f^*$ does not induce a bijection on homotopy classes of functors. 

(As a simple example, take $\cA$ to be the one-object $A_\infty$-category given by the ring $\ZZ/2\ZZ$, and let $\cA'$ be the one-object $A_\infty$-category given by a dg algebra freely resolving $\ZZ/2\ZZ$. If $\cE = \cA'$, one functor space is empty, while the other is not.)

The take-away is that when $\cA$ is not cofibrant, the mapping space $\hom_{\Ainftycat}(\cA,\cE)$ does not have a naive interpretation as a space of honest functors $\cA \to \cE$. Such an interpretation is only available when $\cA$ is cofibrant.
\end{remark}

\begin{remark}[Abundance of cofibrancy]
\label{remark. abundance of cofibrancy}
At the same time, suppose $\cA$ is not cofibrant. We may always find a cofibrant $\cA'$ and a quasi-equivalence $\cA' \xrightarrow{\simeq} \cA$ (Remark~1.8 of~\cite{tanaka-Aoo-units}). We can in fact replace arbitrary morphisms $f: \cA \to \cB$ in $\Ainftycat$ to be between cofibrant objects: Take the pullback 
	\eqnn
	\xymatrix{
	P \ar[r] \ar[d] 
		& \cA' \ar[d]^{\simeq} \\
	\cB \ar[r]^f
		&  \cA
	}
	\eqnnd
and note that the lefthand vertical arrow is an equivalence in $\Ainftycat$, being a pullback of an equivalence. Take a cofibrant replacement $\cB' \xrightarrow{\simeq} P$; then any 2-simplex filling the inner horn $\cB' \to P \to \cB$  (there is a contractible space of such 2-simplices) exhibits an arrow $\cB' \to \cB$ which is an equivalence in $\Ainftycat$. Because $\cB'$ is cofibrant, this edge can be interpreted as an honest functor $\cB' \to \cB$, and is homotopy-invertible; in particular, this functor is a quasi-equivalence. Now, any 2-simplex filling the inner horn $\cB' \to P \to \cA'$ exhibits an edge $f' : \cB' \to \cA'$ replacing $f$. 

In the case that $f$ is a fully faithful inclusion, so is $f'$.

Thus, if one seeks a concrete model for the quotient $\cA/\cB$ in $\Ainftycat$, one may take (up to natural equivalence in $\Ainftycat$) the quotient $\cA'/\cB'$ guaranteed by Theorem~\ref{theorem. quotient models}. 

The necessity for such replacements also appears in the setting of dg-categories~\cite{tabuada-on-dg-quotient}, and in many homotopy-theoretic settings (especially those involving model categories).
\end{remark}

\begin{remark}
If $\cB \to \cA$ is a functor of $A_\infty$-categories, the associated pushout along $\cB \to 0$ is a quotient of $\cA$ along the (full subcategory spanned by the) essential image (Proposition~\ref{prop. quotient is full}). In particular, when discussing quotients, one may as well take $\cB$ to be a full subcategory of $\cA$. 
\end{remark}

\subsection{Some remarks for wrapped Fukaya category users}

Let us close the introduction with a final word of motivation. In the last decade, $\infty$-categories have emerged as powerful tools for proving theorems involving higher homotopy coherences. From the outset, Fukaya categories have required such tools. We hope resources like this paper will empower other mathematicians at the interface of  $\infty$-categorical and Fukaya-categorical phenomena. The utility and need for these kinds of techniques have already appeared, for example, in work of Pardon~\cite{pardon} in combinatorially articulating the (non)-dependence of Floer theory definitions on auxiliary data (such as the choice of Hamiltonians; similar techniques appear in~\cite{gps-covariant}), and in the work~\cite{tanaka-pairing, tanaka-exact} to relate the homotopical richness of Lagrangian cobordisms to Floer-theoretic invariants. In~\cite{oh-tanaka-actions} (which precipitated the writing of the present paper) we apply the results here to show that a Liouville action of a Lie group $G$ on a Liouville sector $M$ induces a homotopy coherent action of $G$ on the wrapped Fukaya category of $M$. This proves a conjecture from Teleman's 2014 ICM address in the Liouville setting.

We also note that almost all constructions of Fukaya categories involve morphism complexes that are homotopically projective (in fact, morphism complexes are typically free as graded modules), so cofibrancy conditions are not restrictive.

We also hope the present work helps to illucidate claims made in other works. For example, \cite{gps-covariant} employs the term localization  without exhibiting the full universal property. The present work justifies the terminology through Theorem~\ref{thm. LM model of localization}.

\subsection{Assumptions and omissions}
We assume the reader is fluent with basic techniques in $\infty$-categories. We refer the reader to an older version of this work on the arXiv for a gentler entry.

\subsection{Acknowledgments}
We would like to thank
Paolo Stellari,
Alberto Canonaco,
and Mattia Ornaghi 
for generously discussing their work with us.
We thank
Gabriel Drummond-Cole and
Rune Haugseng
for helpful conversations,
and the anonymous referee whose comments have improved the paper.
The first author was supported by the IBS project IBS-R003-D1.
The second author was supported by
IBS-CGP in Pohang, Korea
the Isaac Newton Institute in Cambridge, England,
National Science Foundation Grant No. DMS-1440140 while the second author was in residence at the Mathematical Sciences Research Institute in Berkeley, California, in Spring 2019,
an NSF DMS-2044557,
and the Sloan Research Fellowship.

\section{Recollections}
By an $\infty$-category, we mean a simplicial set satisfying the weak Kan condition. Many of the recollections here have detailed accounts in~\cite{tanaka-Aoo-units}.

Let $\Ainftycatt$ denote the usual category of $A_\infty$-categories. Let $W$ be the collection of functors $f: \cA \to \cB$ that induce isomorphisms on cohomology -- i.e., $W$ is the collection of quasi-equivalences. By definition, 
	\eqnn
	\Ainftycat := \Ainftycatt[W^{-1}]
	\eqnnd
is the $\infty$-category obtained by localization along $W$.

\begin{remark}
\label{remark. catt to cat}
In particular, one has a canonical functor from (the nerve of) $\Ainftycatt$ to $\Ainftycat$, so commuting diagrams in $\Ainftycatt$ define diagrams in $\Ainftycat$. This allows us to interpret any functor $f: \cA \to \cB$ as an edge in $\Ainftycat$.
\end{remark}

The converse is not true. That is, diagrams (and the edges) in $\Ainftycat$ do not always have straightforward interpretations in $\Ainftycatt$. However:

\begin{theorem} (Theorem~1.5(b) of~\cite{tanaka-Aoo-units}.)
If $\cA$ is cofibrant (Definition~\ref{defn. cofibrant}), then for any $A_\infty$-category $\cD$, one has a homotopy equivalence
	\eqnn
	\hom_{\Ainftycat}(\cA,\cD)
	\simeq
	N(\fun_{A_\infty}(\cA,\cD))^{\sim}.
	\eqnnd
\end{theorem}
In words: Take the $A_\infty$-nerve of the $A_\infty$-category of unital functors from $\cA$ to $\cD$, then take the largest Kan complex contained in this nerve. The theorem states that this Kan complex is homotopy equivalent to the space of maps from $\cA$ to $\cD$ in $\Ainftycat$.

\begin{defn}
Let $f: \cA \to \cD$ be a functor of $A_\infty$-categories. The {\em essential image} of $f$ is full subcategory of $\cD$ spanned by objects $y \in \cD$ such that, for some object $x \in \cA$, $f(x)$ and $y$ are isomorphic in $H^0(\cD)$. 
\end{defn}

\begin{defn}
We say that a functor $\cA \to \cD$ is fully faithful if for every $x,x' \in \ob \cA$, the chain map $\hom_{\cA}(x,x') \to \hom_{\cD}(fx,fx')$ is an isomorphism on cohomology.
\end{defn}

\section{Stability}
Given an $A_\infty$-category $\cA$, there is a natural pretriangulated completion 
	\eqnn
	\Tw \cA
	\eqnnd
whose definition we do not recall here. This construction is often called the twisted complex construction. The dg case is due to Bondal-Kapranov~\cite{bondal-kapranov-enhanced-triangulated}, while the $A_\infty$-version is due to Kontsevich~\cite{kontsevich-hms}. Accounts can be found in
\cite[Section (3l)]{seidel-book},
\cite[Section 7]{keller-intro-to-algebras-modules}, and
\cite[Chapter 12]{bespalov-lyubashenko-manzyuk} -- in this third work, $\Tw \cA$ is instead written $\cA^{\mathsf tr}$.

\begin{remark}
There is another standard way to take a pretriangulated closure of $\cA$. One considers the Yoneda embedding $\cA \to \cA \Mod$. Noting that $\cA\Mod$ has mapping cones, one may consider the smallest full subcategory inside $\cA\Mod$ containing the image of the Yoneda embedding, containing zero, closed under shifts, and closed under mapping cones. This full subcategory models a triangulated closure of $\cA$.

We will prefer to utilize $\Tw \cA$, for the reason that $\Tw$ is functorial (maps $\cA \to \cB$ induce maps $\Tw \cA \to \Tw \cB$ in a way respecting composition on the nose) and admits a natural transformation $\cA \to \Tw \cA$. (In contrast, it is rather delicate to make the Yoneda embedding covariantly functorial without constructing, say, a formula-based model of left Kan extensions in the setting of $A_\infty$-categories, which we do not do here.)
\end{remark}

\begin{prop}
\label{prop. Tw respects quasieqs}
Let $f: A \to B$ be a quasi-equivalence of $A_\infty$-categories. Then the induced functor $\Tw f : \Tw A \to \Tw B$ is also a quasi-equivalence.
\end{prop}

\begin{proof}
Observe that any object in $\Tw\cA$ admits a finite filtration whose associated gradeds are (shifts of) objects of $\cA$. If $f$ is a quasi-equivalence, long exact sequences of cohomology groups of hom-complexes assure that $\Tw f$ is a quasi-equivalence.
\end{proof}

The following are also well-known, and their verifications can be found at the references below, or follow straight from the definitions.

\begin{prop}
\label{prop. Tw facts}
\enum[(a)]
\item $\Tw$ extends to morphisms to define a functor $\Ainftycatt \to \Ainftycatt$. 

\item\label{item. A to Tw A} Moreover, one has a natural (in the $\cA$ variable) functor $u_{\cA}: \cA \to \Tw \cA$ (Proposition~12.3 of~\cite{bespalov-lyubashenko-manzyuk}.)

\item For every $\cA$, $u_\cA$ is fully faithful.

\item For any $\cA$, the two natural functors $\Tw \cA \to \Tw \Tw \cA$ -- one is $u_{\Tw\cA}$, while the other applies $\Tw$ to $u_{\cA}$ -- are both quasi-equivalences (Proposition~12.11 of~\cite{bespalov-lyubashenko-manzyuk}). 

\item
\label{item. Tw of Fun is functors into Tw} There is a natural functor $\Tw \fun_{A_\infty}(\cA,\cB) \to \fun_{A_\infty}(\cA,\Tw\cB)$ and this functor is an equivalence of $A_\infty$-categories. (Section~10.47 and Proposition~11.44 of~\cite{bespalov-lyubashenko-manzyuk}.)

\enumd
\end{prop}

\begin{remark}\label{remark. Tw idempotent}
By Proposition~\ref{prop. Tw respects quasieqs} and 
by functoriality of localization, $\Tw$ induces a functor of $\infty$-categories
	\eqnn
	\Tw: \Ainftycat \to \Ainftycat
	\eqnd
where we abuse the notation.
Moreover, by the final claim in Proposition~\ref{prop. Tw facts}, $\Tw$ is an idempotent functor of $\infty$-categories, so is a localization of $\infty$-categories -- i.e., $\Tw$ admits a fully faithful right adjoint (given by the full inclusion of the essential image of $\Tw$ -- see Proposition~5.2.7.4 of~\cite{htt}).

We denote this full subcategory by 
	\eqnn
	\Ainftycat^{\stable} \subset \Ainftycat.
	\eqnd
We call this the full subcategory of {\em pretriangulated} or {\em stable} $A_\infty$-categories. 
\end{remark}

\begin{remark}\label{remark. universal property of Tw}
Let $\cA$ be an arbitrary $A_\infty$-category and $\cD$ a stable $A_\infty$-category. Then, by adjunction, the map of mapping spaces
	\eqnn
	\hom_{\Ainftycat}(\Tw\cA,\cD) \to
	\hom_{\Ainftycat}(\cA,\cD)
	\eqnd
is a homotopy equivalence. Informally, one can interpret the universal property of $\cA \to \Tw\cA$ as follows: If $\cD$ is stable, then any functor $f: \cA \to \cD$ extends (uniquely up to contractible choice) to a functor from $\Tw\cA$:
	\eqnn
	\xymatrix{
	\cA \ar[d] \ar[r]^f & \cD \\
	\Tw\cA \ar@{-->}[ur]_{\exists !} &.
	}
	\eqnd
\end{remark}

\begin{remark}
\label{remark. every functor is exact}
Any functor between $A_\infty$-categories preserves mapping cones. 
Indeed, given $f: \cA \to \cD$, consider the induced functor $\Tw f : \Tw \cA \to \Tw \cD$. By definition of $\Tw$, for any closed degree zero morphism $\alpha: X \to Y$ in $\cA$, we have a natural identification $(\Tw f) (\cone(\alpha)) \cong \cone(\Tw f(\alpha)) \cong \cone(f(\alpha))$. Now suppose that $\cone(\alpha)$ is in the essential image of $\cA \to \Tw \cA$ -- so $\cone(\alpha)$ is isomorphic in $H^0 \cA$ to some $Q \in \ob \cA$. By the naturality of $\cA \to \Tw \cA$, it follows that $\Tw(f)(\cone(\alpha))$ is isomorphic to $f(Q)$ in $H^0(\cD)$. In other words, any mapping cone of $\alpha$ is sent by $f$ to a mapping cone of $f(\alpha)$. 

Even better, the natural exact sequence $X \to Y \to \cone(\alpha)$, together with the natural homotopy of the composition to the zero map, is sent to the sequence $f(X) \to f(Y) \to f(\cone(\alpha))$ and a homotopy from this composition to zero. 

A choice of data realizing the the isomorphism $\cone(\alpha) \cong Q \in H^0 \cA$  induces a sequence $X \to Y \to Q$ (together with a null homotopy of this composition to the 0 map) in $\cA$, and $f$ sends this to a sequence $f(X) \to f(Y) \to f(Q)$  (with a null homotopy of this composition). The sequences in $\cA$ and in $\cD$ are mapping cone sequences in the sense that for any test object $w \in \ob \cA$, the sequence of cochain complexes induced by $\hom_{\cA}(w,-)$ is a mapping cone sequence in the usual dg-category of cochain complexes (and likewise for any $w \in \ob \cD$). 

The previous paragraph articulates the precise sense in which $f$ preserves the mapping cone sequences that already exist in $\cA$.
\end{remark}

\begin{defn}
We say that a functor is {\em exact} if it preserves mapping cone sequences. 
\end{defn}

By Remark~\ref{remark. every functor is exact}, every functor between $A_\infty$-categories is exact; regardless, the adjective is worth having for use in proofs (and is commonly used in higher algebra, especially in contexts where not every functor is exact).

\section{Quotients}

\begin{definition}[Quotients]\label{defn. Aoo quotient}
Fix an $A_\infty$-category $\cA$.
Let $\cB \subset \cA$ be a full subcategory, and let $0$ be the zero category (having a single object and only the 0 morphism). Then the {\em quotient} $A_\infty$-category $\cA / \cB$ is defined to be the pushout 
in $\Ainftycat$:
	\eqn\label{eqn. pushout diagram}
	\xymatrix{
	\cB \ar[r] \ar[d] & \cA \ar[d] \\
	0 \ar[r] & \cA/\cB
	}
	\eqnd
\end{definition}

Quotients always exist because $\Ainftycat$ is presentable (and in particular has all colimits, including pushouts). See Remark~3.25 of~\cite{tanaka-Aoo-units}.

The following is a generalization of Definition~6.4 of~\cite{lyubashenko-ovsienko}:

\begin{defn}
Let $f: \cA \to \cD$ be a functor and $i: \cB \to \cA$ a functor. If, for every $X,Y \in \ob \cB$, the chain map $\hom_{\cA}(X,Y) \to \hom_{\cD}(fX,fY)$ is null-homotopic, we will say that $f$ is {\em contractible along $i$.}
When $i$ is the inclusion of a full subcategory, we say $f$ is contractible along $\cB$.
\end{defn}

\begin{remark}
\label{remark. contractible is property of objects}
When all $A_\infty$-categories and functors are unital (as we assume throughout this work), the notion of $f$ being contractible along $i$ has a simple characterization articulable at the level of objects. If $f$ is  contractible along $i$, for every object $Z$ in the image of $fi$, the null homotopies exhibit a homotopy between a unit $e_Z$ of $Z$ and $0 \in \hom_{\cD}(Z,Z)$. On the other hand, suppose that $f$ is a functor for which every object of $\cB$ is sent under $fi$ to an object $Z$ whose unit is cohomologous to zero. Then (because $m^2(e_Z,-)$ is homotopic to the identity) we see that any chain map to $\hom_{\cD}(X,Z)$ is null-homotopic (for any object $X \in \cD$). In particular, $f$ is automatically contractible along $i$.

In short: If $f$ and $i$ are unital, $f$ is contractible along $i$ if and only if the image of $fi$ consists of objects whose units are null-cohomologous.
\end{remark}

We first record that being contractible along $i$ is a property that manifests in $\Ainftycat$, and can be generalized to arbitrary edges in $\Ainftycat$.

\begin{prop}
\label{prop. factoring through zero is a property}
Fix an arbitrary edge $i: \cB \to \cA$ in $\Ainftycat$. Then:
\enum[(a)]
\item for any edge $f: \cA \to \cD$ in $\Ainftycat$, the space of diagrams $\Delta^1 \times \Delta^1 \to \Ainftycat$ of the form
	\eqn\label{eqn. quotient diagram}
	\xymatrix{
	\cB \ar[r]^{i} \ar[d] & \cA \ar[d]^f \\
	0 \ar[r] & \cD
	}
	\eqnd
is either empty or contractible.
\item\label{item. factoring through zero is contractibility}
The space of such diagrams is non-empty if and only if, for any cofibrant replacement $i' : \cB' \to \cA'$ (in the sense of Remark~\ref{remark. abundance of cofibrancy}), any map $f': \cA' \to \cD$ obtained by filling the inner horn $\cA' \to \cA \xrightarrow{f} \cD$ is contractible along $i'$.
\item\label{item. classical contractible implies oo cat contractibility} Suppose that both $i$ and $f$ are in the image of the functor $\Ainftycatt \to \Ainftycat$ (so that $i$ and $f$ can be interpreted as honest functors of $A_\infty$-categories -- see Remark~\ref{remark. catt to cat}). If $f$ is contractible along $i$, then the space of diagrams of the form~\eqref{eqn. quotient diagram} is non-empty.
\enumd
\end{prop}

\begin{proof}
Concretely, the space of diagrams~\eqref{eqn. quotient diagram} is the fiber of the forgetful map
	\eqnn
	\fun(\Delta^1 \times \Delta^1, \Ainftycat)^{\sim}
	\to
	\fun(\Delta^1 \bigcup \Lambda^2_1, \Ainftycat)^{\sim}
	\eqnd
above the horn given by $\cB \xrightarrow{i} \cA \xrightarrow{f} \cD$ and the edge given by the image of the unique map $\cB \to 0$ in $\Ainftycatt$ (see Remark~\ref{remark. catt to cat}). This fiber is homotopy equivalent to the space of tuples $(z,T_1,T_2)$ where
	\begin{itemize}
	\item $z$ is an edge from $0$ to $\cD$.
	\item $T_1$ is a 2-simplex with $\partial_2$ and $\partial_0$ given by (the image of the unique functor) $\cB \to 0$ and $z$, respectively.
	\item $T_2$ is a 2-simplex with inner horn given by $\cB \xrightarrow{i} \cA \xrightarrow{f} \cD$, and
	\item $\partial_1 T_1 = \partial_1 T_2$.
\end{itemize}
Given an inner horn, there exists a unique (up to contractible choice) 2-simplex filling the horn (Corollary 2.3.2.2 of~\cite{htt}). Thus, let us fix a choice of $T_2$. Then the space of tuples $(z,T_1,T_2)$ is homotopy equivalent to the space of pairs $(z,T_1)$, with $T_1$ satisfying the boundary conditions above. On the other hand,
because $0$ is cofibrant, the space of choices of $z$ can be interpreted as the space of functors from $0$ to $\cD$.
This space can be computed as the nerve of the $A_\infty$-category of $A_\infty$-functors from $0$ to $\cD$, and this nerve is particular easy to characterize:
The nerve has a vertex for every object of $\cD$ for whom $0$ is a unit, and a unique $k$-simplex for any ordered $(k+1)$-tuple of such objects, with the obvious face and degeneracy maps. In particular, it is contractible. Thus, the space of pairs $(z,T_1)$ fibers over the contractible space of $z$. We conclude that the space of pairs $(z,T_1)$ is (fixing a $z$) homotopy equivalent to the space of $T_1$ satisfying the following boundary conditions: Its inner horn is $\cB \to 0 \xrightarrow{z} \cD$, and $\partial_2 T_1 = f|_{\cB}$.
(Here, we have denoted $\partial_2 T_2$ by the suggestive notation $f|_{\cB}$.)
 
 By the usual $\infty$-categorical yoga, let us fix an arbitrary 2-simplex $T$ filling the inner horn $\cB \to 0 \xrightarrow{z} \cD$, and suggestively denote $\partial_2 T = z|_{\cB}$. Then the space of $T_1$ (with the above boundary conditions) is homotopy equivalent to the space of homotopies $h$ from $z|_{\cB}$ to $f|_{\cB}$. Under our assumption that there is already one such path, we are thus left to compute the based loop space of $\hom_{\Ainftycat}(\cB,\cD)$, based at a fixed $z|_{\cB}$. 

Now we finally invoke that we are working in the $\infty$-category $\Ainftycat$. 
If $\cB' \to \cB$ is a quasi-equivalence from a cofibrant $\cB'$, then we have a homotopy equivalence $\hom_{\Ainftycat}(\cB,\cD) \simeq \hom_{\Ainftycat}(\cB',\cD)$ -- so we may as well assume $\cB$ cofibrant.
We may compute the homotopy groups $\pi_{i+1}, i \geq 0$ of $\hom_{\Ainftycat}(\cB,\cD)$ based at $z|_{\cB}$ by computing the cohomology of the complex of natural transformations:
	\eqnn
	H^{-i}\hom_{\fun_{A_\infty}(\cB,\cD)}(z|_\cB,z|_\cB).
	\eqnd
(See Corollary~1.9 of~\cite{tanaka-Aoo-units}.) Parsing the definitions, we may replace $\cD$ by its full subcategory $\cD_{null}$ of objects whose units are null in cohomology. Then the functor $\cD_{null} \to 0$ is a quasi-equivalence. So $\fun_{A_\infty}(\cB,0)$ is quasi-equivalent to $\fun_{A_\infty}(\cB,\cD_{null})$. In particular, the above cohomology groups all vanish, showing that the homotopy groups of $\hom_{\Ainftycat}(\cB,\cD)$ based at $z|_{\cB}$ vanish. 

Let us now prove~\eqref{item. factoring through zero is contractibility}. Cofibrant replacement of $i: \cA \to \cB$, in the sense of Remark~\ref{remark. abundance of cofibrancy}, leaves the space of \eqref{eqn. quotient diagram} unchanged up to homotopy equivalence. Henceforth we will thus assume $\cA,\cB$ to be cofibrant. 

If the space of~\eqref{eqn. quotient diagram} is non-empty, choose a functor $z: 0 \to \cD$ and a homotopy $h$ from $f|_{\cB}$ to $z|_{\cB}$.
Because $\cB$ is cofibrant, this can be represented by a degree -1 cochain in the complex of natural transformations from  $f|_{\cB}$ to $z|_{\cB}$ (Theorem~1.5(b) of~\cite{tanaka-Aoo-units}). Such a cochain in particular exhibits, for every $X,Y \in \ob \cB$, 
\begin{itemize}
\item Degree 0 closed elements $\eta_X \in \hom_{\cD}(fX,zX)$ and $\eta_Y \in \hom_{\cD}(fY,zY)$, both of which are  isomorphisms in $H^* \cD$, and
\item A chain homotopy making the following diagram of cochain complexes commute:
	\eqnn
	\xymatrix{
	\hom_{\cB}(X,Y) \ar[r]^{f} \ar[d]^{z}
		& \hom_{\cD}(fX,fY) \\
	\hom_{\cD}(zX,zY) \ar[ur]_{(\eta_X)^\ast(\eta_Y^{-1})_\ast}
	}.
	\eqnd
\end{itemize}
Because $z$, at the level of morphism chain complexes, sends all elements to zero, the above homotopy-commuting triangle shows that $f$ is null-homotopic on all mapping complexes.

For the converse, let $\cD_f \subset \cD$ be the full subcategory spanned by the image of $f|_{\cB}$. (We continue to safely assume $\cA,\cB$ cofibrant -- after possibly replacing them -- so that the edge $f$ in $\Ainftycat$ has an interpretation as an honest $A_\infty$-functor.) We then have a diagram $\Delta^1 \times \Delta^1 \to \Ainftycat$ which we draw as
	\eqnn
	\xymatrix{
	\cB \ar[r]^{i} \ar[d] & \cA \ar[d]^f \\
	\cD_f \ar[r]^{\subset} & \cD.
	}
	\eqnd
Further, any functor $0 \to \cD_f$ is a quasi-equivalence, because all hom complexes $\hom_{\cD_f}(-,-)$ are contractible (and, as a result, all objects in $\cD_f$ are isomorphic in $H^0\cD_f$) by hypothesis. Thus, one finds a diagram $\Delta^2 \bigcup (\Delta^1 \times \Delta^1) \to \Ainftycat$ of the form
	\eqnn
	\xymatrix{
	&\cB \ar[r]^{\subset} \ar[d] \ar@{-->}[dl]_{\exists}& \cA \ar[d]^f \\
	0 \ar[r]^{\simeq} &\cD_f \ar[r]^{\subset} & \cD.
	}
	\eqnd
(In classical language, one should interpret the triangle as commuting up to homotopy.) It is now straightforward to exhibit a diagram~\eqref{eqn. quotient diagram} -- say, by filling horns appropriately.

Now we prove~\eqref{item. classical contractible implies oo cat contractibility}.
Choose an object $X \in \ob \cB$ and define $z: 0 \to \cD$ to be the unique functor sending the object of 0 to $fiX$. Now recall the following (somewhat confusing) fact: Any (possibly non-unital) $A_\infty$-functor $\phi: \cX \to \cY$ admits a zero natural transformation $T$ from any ``zero functor'' -- i.e., any functor $\psi_Y: \cX \to \cY$ that sends every $X \in \ob \cX$ to some fixed object $Y \in \ob \cY$ and sends every element of $\hom_{\cX}$ to zero (and whose higher arity terms are all zero). The zero natural transformation $T$ likewise is zero at every arity. If we further assume that $\phi$ is a (unital) functor that satisfies the property that the hom chain maps are all null-homotopic, Proposition~7.15 of~\cite{lyubashenko-category-of} or guarantees that the zero natural transformation $T$ is in fact a natural equivalence. 

In particular, $fi$ is naturally equivalent to the composition $\cB \to 0 \xrightarrow{z} \cD$. Now choosing a cofibrant replacement $q: \cB' \to \cB$, the pullback map
	\eqnn
	q^* \fun_{A_\infty}(\cB,\cD) \to 
	\fun_{A_\infty}(\cB',\cD)
	\eqnnd
induces a chain map on the complex of (pre)natural transformations; in particular,  $\cB' \to \cB \to 0 \xrightarrow{z} \cD$ is naturally equivalent to $\cB' \to \cB \to \cA \to \cD$. Because $\cB'$ is cofibrant, this exhibits a homotopy in $\Ainftycat$, and in particular a diagram
	\eqnn
	\xymatrix{
	\cB' \ar[r] \ar[d] & \cA \ar[d]\\
	0 \ar[r] & \cD
	}
	\eqnnd
which, by transferring along the equivalence $q$, exhibits a diagram~\eqref{eqn. quotient diagram}.
\end{proof}

\begin{proposition}\label{prop. pushout diagram is automatic for quotients}
Let $\cA$ be an $A_\infty$-category and fix $\cB \subset \cA$ a full subcategory. Fix further an $A_\infty$-category $\cQ$ along with a functor $\cA \to \cQ$. Then the following are equivalent:
\enum[i.]
\item\label{item. full subcat send to zero}
For any $A_\infty$-category $\cD$, and for any cofibrant replacement $\cA' \xrightarrow{\simeq} \cA$, the restriction map $\hom_{\Ainftycat}(\cQ, \cD) \to \hom_{\Ainftycat}(\cA', \cD)$ (induced by restriction along  $\cA' \to \cA \to \cQ$) is an injection on $\pi_0$ and a homotopy equivalence along each component. Further, the restriction map has essential image spanned by those functors that are contractible along $\cB$.

\item\label{item. Q completes to pushout} One can complete $\cA \to \cQ$ to a pushout diagram in $\Ainftycat$
	\eqnn
	\xymatrix{
	\cB \ar[r] \ar[d] & \cA \ar[d] \\
	0 \ar[r] & \cQ
	}
	\eqnd
where the top horizontal arrow is the inclusion $\cB \subset \cA$ and the bottom horizontal arrow is to a zero object of $\cQ$.
\enumd	
\end{proposition}

\begin{remark}
\label{remark. cofibrancy identification of functors}
The description of the image in~\eqref{item. full subcat send to zero} only makes sense if we have an interpretation of $\hom_{\Ainftycat}(\cA,\cD)$ as having vertices given by $A_\infty$-functors $\cA \to \cD$. Such an interpretation is only available a priori if $\cA$ is cofibrant.
See Remark~1.6 and Theorem~3.37 of~\cite{tanaka-Aoo-units}. Note also we need only characterize the image of the restriction map, so there is no need for $\cQ$ to be cofibrant to state~\eqref{item. full subcat send to zero}.
\end{remark}

\begin{proof}[Proof of Proposition~\ref{prop. pushout diagram is automatic for quotients}.]
As usual we assume $\cA$ is cofibrant (replacing if necessary). 
By Proposition~\ref{prop. factoring through zero is a property},
in either of the cases
\ref{item. full subcat send to zero}. or \ref{item. Q completes to pushout}. of the proposition,
we can choose a diagram in $\Ainftycat$ exhibiting a homotopy between $\cB \to 0 \to \cQ$ and $\cB \to \cA \to \cQ$. Concretely, this is a simplicial set map 
	\eqn\label{eqn. J base square} 
	J: \Delta^1 \times \Delta^1 \to \Ainftycat.
	\eqnd
When proving either implication of the proposition, we fix a choice of $J$. For the time being, let us also fix a choice of $A_\infty$-category $\cD$.

We have the following commutative diagram of simplicial sets:
\eqn\label{eqn. big diagram}
\begin{tikzpicture}[scale=1.2]
	\node (S1) at (0,0)
		{
    	$\left\{ 
			\cQ \to \cD
    	\right\}$
		}
	;
	\node (S2) at (4,0) 
		{
    	$\left\{ 
    		\begin{tikzpicture}[scale=1.5, baseline={([yshift=-.8ex]current bounding box)}]
                \node (A) at (1,1) {$\cA$};
                \node (Q) at (1,0) {$\cQ$};
                \node (D) at (2,-1) {$\cD$};
                \path[->,font=\scriptsize,tips]
                    (A) edge node[left]{$\bullet$} (Q)
                    (A) edge (D)
                    (Q) edge (D)
                    ;
                \end{tikzpicture}
    	\right\}$
		}
	;
	\node (S3) at (8,0) 
		{		
			$\left\{
    		\begin{tikzpicture}[scale=1.5, baseline={([yshift=-.8ex]current bounding box)}]
                \node (A) at (1,1) {$\cA$};
                \node (D) at (2,-1) {$\cD$};
                \path[->,font=\scriptsize,tips]
                    (A) edge (D)
                    ;
                \end{tikzpicture}
             \right\}$
		}
	;
	\node (R2) at (4,4) 
		{
			$\left\{
    		\begin{tikzpicture}[scale=1.5, baseline={([yshift=-.8ex]current bounding box)}]
                \node (B) at (0,1) {$\cB$};
                \node (A) at (1,1) {$\cA$};
                \node (O) at (0,0) {$0$};
                \node (Q) at (1,0) {$\cQ$};
                \node (D) at (2,-1) {$\cD$};
                \path[->,font=\scriptsize,tips]
                    (B) edge node[above]{$\bullet$} (A)
                    (B) edge node[left]{$\bullet$} 	(O)
                    (O) edge node[above]{$\bullet$} (Q)
                    (A) edge node[left]{$\bullet$} (Q)
                    (A) edge (D)
                    (Q) edge (D)
                    (O) edge (D);
                \end{tikzpicture}
             \right\}$
		}
	;
	\node (R3) at (8,4) 
		{		
			$\left\{
    		\begin{tikzpicture}[scale=1.5, baseline={([yshift=-.8ex]current bounding box)}]
                \node (B) at (0,1) {$\cB$};
                \node (A) at (1,1) {$\cA$};
                \node (O) at (0,0) {$0$};
                \node (D) at (2,-1) {$\cD$};
                \path[->,font=\scriptsize,tips]
                    (B) edge node[above]{$\bullet$} (A)
                    (B) edge node[left]{$\bullet$} 	(O)
                    (A) edge (D)
                    (O) edge (D);
                \end{tikzpicture}
             \right\}$
		}
	;
    \path[->,font=\scriptsize,tips]
    (R2) edge node[left]{$\sim$}  node[right]{\eqref{item. XU and XV equivalences}1}
    	(S1)
    (R2) edge node[above]{\eqref{item. Q completes to pushout}} 
    	(R3)
    (R2) edge node[right]{$\sim$} node[left]{\eqref{item. XU and XV equivalences}2}
    	(S2)
    (R3) edge node[right]{f.f.} node[left]{\eqref{eqn. fully faithful YW}}
    	(S3)
    (S2) edge node[above]{$\sim$} node[below]{\eqref{item. VU}} 
    	(S1)
    (S2) edge node[above]{\eqref{item. full subcat send to zero}} 
    	(S3);
\end{tikzpicture}
\eqnd
To explain the ingredients, and to reduce clutter, let us give each simplicial set in the diagram a name:
\eqn\label{eqn.smaller diagram}
\begin{tikzpicture}[scale=1.5]
	\node (S1) at (0,0) {$U$};
	\node (S2) at (2,0) {$V$};
	\node (S3) at (4,0) {$W$};
	\node (R2) at (2,2) {$X$};
	\node (R3) at (4,2) {$Y$};
    \path[->,font=\scriptsize,tips]
    (R2) edge node[left]{$\sim$}  node[right]{\eqref{item. XU and XV equivalences}1}
    	(S1)
    (R2) edge node[above]{\eqref{item. Q completes to pushout}} 
    	(R3)
    (R2) edge node[right]{$\sim$} node[left]{\eqref{item. XU and XV equivalences}2}
    	(S2)
    (R3) edge node[right]{f.f.}  node[left]{\eqref{eqn. fully faithful YW}}
    	(S3)
    (S2) edge node[above]{$\sim$} node[below]{\eqref{item. VU}} 
    	(S1)
    (S2) edge node[above]{\eqref{item. full subcat send to zero}} 
    	(S3);
\end{tikzpicture}
\eqnd
Then
\begin{enumerate}[(A)]
	\item $U$ is the simplicial set of all simplicial set maps $\Delta^1 \to \Ainftycat$ such that the vertices of $\Delta^1$ are sent to $\cQ$ and $\cD$, respectively. So for example, the $k$-simplices of $U$ are given by the collection of maps $f: \Delta^k \times \Delta^1 \to \Ainftycat$ for which $f|_{\Delta^k \times d_1 \Delta^1}$ is the degenerate $k$-simplex at $\cQ$, and $f|_{\Delta^k \times d_0 \Delta^1}$ is the degenerate $k$-simplex at $\cD$. This explains why we have used the informal notation $\{\cQ \to \cD\}$ to denote $U$ in \eqref{eqn. big diagram}. ($U$ is the simplicial set constructed out of ``the collection of all edges from $\cQ$ to $\cD$''.) 
	\item  $V$ is the simplicial set of all simplicial set maps $\Delta^2 \to \Ainftycat$ such that the edge $d^2 \Delta^2$ is sent to the map $\cA \to \cQ$ fixed in the hypotheses. The $k$-simplices of $V$ are given by simplicial set maps $f: \Delta^k \times \Delta^2 \to \Ainftycat$ satisfying the natural degeneracy conditions along $f|_{\Delta^k \times \Delta^0}$ (to fix $\cD$) and along $f|_{\Delta^k \times d^2 \Delta^2}$. In \eqref{eqn. big diagram}, the symbol $\bullet$ along the edge $\cA \to \cQ$ above is meant to indicate that $d_2 \Delta^2$ moves in degenerate families.
	\item\label{item. VU} Note that there is a natural ``forgetful'' map of simplicial sets from $V$ to $U$ by restricting from $\Delta^2$ to $d^0 \Delta^2$.  Because the inclusion of $d^0 \Delta^2 \cup d^2 \Delta^2 \into \Delta^2$ is inner anodyne, and because $\Ainftycat$ is an $\infty$-category, this forgetful map is a trivial inner fibration. Thus the forgetful map (as indicated in \eqref{eqn.smaller diagram}) is a homotopy equivalence.
	\item $W$ is the simplicial set of all edges from $\cA$ to $\cD$. The map of simplicial sets indicated by \eqref{item. full subcat send to zero} is the forgetful map. 
	\item $X$ is as follows. Let $\Delta^3 \cup_{d^1 \Delta^3} \Delta^3$ denote the simplicial set given by gluing two 3-simplices along the 1st face, $d^1 \Delta^3$, of each 3-simplex. Then $X$ is the simplicial set  of maps $f: \Delta^3 \cup_{d^1 \Delta^3} \Delta^3 \to \Ainftycat$ such that the restriction of $f$ to $\Delta^1 \times \Delta^1 \subset \Delta^3 \cup_{d^1 \Delta^3} \Delta^3$ is equivalent to the map $J$ from~\eqref{eqn. J base square}. As before, there are natural degeneracy conditions to articulate what we mean by a $k$-simplex of $X$, and the $\bullet$ symbols in \eqref{eqn. big diagram} are meant to indicate that the value of $f$ is fixed along $J$. 
	\item\label{item. XU and XV equivalences} $X$ thus has a forgetful map \eqref{item. XU and XV equivalences}1 to $U$, and a forgetful map \eqref{item. XU and XV equivalences}2 to $V$. Both forgetful maps are homotopy equivalences again by an inner anodyne argument.
	\item\label{eqn. fully faithful YW} Finally, $Y$ is the simplicial set of maps $\Delta^1 \times \Delta^1 \to \Ainftycat$ such that the two initial edges of $\Delta^1 \times \Delta^1$ are sent to the maps $\cB \to \cA$ and $\cB \to 0$, while the terminal vertex of $\Delta^1 \times \Delta^1$ is sent to $\cD$.  The forgetful map to $W$ is an injection on $\pi_0$ and an equivalence by Proposition~\ref{prop. factoring through zero is a property}. (The``f.f.'' in the diagrams stands for fully faithful.) Importantly, the image of $Y \to W$ is the space of functors characterized in~\ref{item. full subcat send to zero}., again by Proposition~\ref{prop. factoring through zero is a property}.
\end{enumerate}

Importantly, all the arrows above are natural in the $\cD$ variable. This follows from standard constructions: All five simplicial sets $U, V, W, X, Y$ are the fibers (over $\cD$) of a coCartesian fibration over $\Ainftycat$. 

Suppose \ref{item. full subcat send to zero}. holds. We note that choosing an inverse to the arrow \eqref{item. VU} in the diagram, then composing with the arrow labeled \eqref{item. full subcat send to zero}, is a model for the restriction map $\hom_{\Ainftycat}(\cQ,\cD) \to \hom_{\Ainftycat}(\cA,\cD)$. Thus assumption \ref{item. full subcat send to zero}. implies that the arrow \eqref{item. full subcat send to zero} is a fully faithful inclusion (i.e., an injection on $\pi_0$ and a homotopy equivalence on each connected component). Thus, the arrow labeled \eqref{item. Q completes to pushout} is a homotopy equivalence on each connected component because all the other arrows in the commutative diagram are. By the same reasoning, this arrow is an injection on $\pi_0$. Moreover, we claim the arrow labeled \eqref{item. Q completes to pushout}  is a surjection on $\pi_0$. To see this, suppose $f$ is a vertex of $Y$ and let $f': \cA \to \cD$ be its image in $W$. By hypothesis \ref{item. full subcat send to zero}., we know that the image of the arrow \eqref{item. full subcat send to zero} consists precisely of such $f'$. This completes the proof that \eqref{item. Q completes to pushout} is a surjection on $\pi_0$.

Because \eqref{item. Q completes to pushout} is a bijection on $\pi_0$ and a homotopy equivalence on connected components, it is a homotopy equivalence. This exhibits $J$ as a pushout diagram, hence \ref{item. Q completes to pushout}. is proven.
 
Now suppose \ref{item. Q completes to pushout}. then the arrow \eqref{item. Q completes to pushout} is a homotopy equivalence. In particular (by the commutativity of the diagram) the arrow \eqref{item. full subcat send to zero} is an injection on $\pi_0$ and a homotopy equivalence on each component. It remains to identify the image of \eqref{item. full subcat send to zero} on $\pi_0$ as precisely those maps $\cA \to \cD$ that are null-homotopic once pre-composed with $\cB \to \cA$. But this is precisely the image of the righthand vertical map \eqref{eqn. fully faithful YW}, again by Proposition~\ref{prop. factoring through zero is a property}. Thus \ref{item. full subcat send to zero}. holds.

\end{proof}

One can also compute pushouts~\eqref{eqn. pushout diagram} when $\cB \to \cA$ is not a fully faithful inclusion. 

\begin{prop}\label{prop. quotient is full}
If $i: \cB \to \cA$ is a functor of $A_\infty$-categories, the pushout along $\cB \to 0$ is a quotient of $\cA$ along the full subcategory spanned by the essential image of $i$.
\end{prop}

\begin{proof}
Let $\cB^\#$ denote the full subcategory of $\cA$ spanned by the essential image of $i$. Then for any functor $f: \cA \to \cD$, pullback along $\cB \to \cB^\#$ induces a map $\alpha_f$ from the space of diagrams of the form
	\eqnn
	\xymatrix{
	\cB^\# \ar[r]^{i} \ar[d] & \cA \ar[d]^f \\
	0 \ar[r] & \cD
	}
	\eqnnd
to the space of diagrams~\eqref{eqn. quotient diagram}. But for a fixed $f$, these spaces are either empty or contractible by~\eqref{prop. factoring through zero is a property} -- and emptiness is detected purely by the essential image of $i$ (Remark~\ref{remark. contractible is property of objects}). That is, for every $f$, $\alpha_f$ is a homotopy equivalence. 

By definition, a quotient $\cQ$ of $\cB \to \cA$ corepresents diagrams of the form~\eqref{eqn. quotient diagram}.  (For example, the space of maps from $\cQ \to \cD$ is homotopy equivalent to the space of diagrams of the form~\eqref{eqn. quotient diagram} by the homotopy equivalence diagram~\eqref{eqn.smaller diagram}.) Letting $\cQ^\#$ denote the quotient of $\cB^\# \to \cA$, there is a natural map $\cQ \to \cQ^\#$, and the induced pullback of corepresenting functors has components given by $\alpha_f$. Because these components are all homotopy equivalences by the previous paragraph, the Yoneda Lemma tells us that $\cQ \to \cQ^\#$ is an equivalence.
\end{proof}

\subsection{Models for quotients}

\begin{notation}[$\sD$ and $\sq$]
Fix an $A_\infty$-category $\cA$ and a full subcategory $\cB$.
There are various constructions of quotients in the literature, and we hone in on two. The first we will denote by $\sD(\cA|\cB)$ following Lyubashenko-Ovsienko~\cite{lyubashenko-ovsienko}, and the second by $\sq(\cA|\cB)$ following Lyubashenko-Manzyuk~\cite{lyubashenko-manzyuk}. We do not recall their full definitions here and refer the reader to the papers just cited. 
\end{notation}

\begin{remark}
The localization used in~\cite{gps-covariant} is modeled on $\sD$. One nice feature of $\sD$ is that morphism complexes of a localization are easily shown to be sequential colimits (Lemma~\ref{lemma. local sequences localize}). 
Regardless, the models $\sD$ and $\sq$ are not necessary to prove many properties one needs of, say, the wrapped Fukaya category (even for applications in \cite{oh-tanaka-actions}). For example, $\kk$-linear versions of arguments used in ~I.3 of~\cite{nikolaus-scholze} suffice to prove Lemma~\ref{lemma. local sequences localize}. 
 
\end{remark}

\begin{remark}
In principle, the equivalence $\Ainftycat \simeq \dgcat$ means we can model quotients of $A_\infty$-categories using their Yoneda embeddings, then resorting to a dg quotient. For example, the construction of Drinfeld~\cite{drinfeld-dg-quotients} together with the universal property verified by Tabuada~\cite{tabuada-dg-quotients} means we may construct quotients by considering the image $\cA \subset \cA\Mod$, finding an appropriate replacement of this image, and constructing a quotient. But having multiple models is useful. 

The model of Lyubashenko-Ovisienko~\cite{lyubashenko-ovsienko} is particularly useful when the base commutative ring $\kk$ is a field, but it is not straightforward to prove its universal property directly. For example, an inconvenience presents itself when trying to carry out the previous paragraph, which is that the Yoneda embedding of a homotopically flat $A_\infty$-category need not itself be a homotopically flat dg-category. One issue is that arbitrary direct products of homotopically projective complexes need not be homotopically projective.
\end{remark}

\begin{theorem}[Theorem 1.3 of \cite{lyubashenko-manzyuk}]\label{theorem. lyubashenko-manzyuk universal property}
For any $A_\infty$-category $\cA$ and any full subcategory $\cB \subset \cA$, there is a functor $q: \cA \to \sq(\cA|\cB)$ satisfying the following property: For any $A_\infty$-category $\cD$, restriction along $q$ induces a fully faithful embedding of $A_\infty$-categories
	\eqnn
	\fun_{A_\infty}(\sq(\cA|\cB),\cD)
	\xra{q^*}
	\fun_{A_\infty}(\cA, \cD)
	\eqnd
whose essential image consists of functors $f$ that are contractible along $\cB$. Even better, if one treats $q^*$ as having codomain given by this essential image, $q^*$ admits an inverse up to natural equivalence of functors.
\end{theorem}

\begin{theorem}[Theorem 7.4 of \cite{lyubashenko-manzyuk}]\label{theorem. lyubashenko-manzyuk different quotients agree}
There is a functor $\cA \to \sD(\cA|\cB)$ which is contractible along $\cB$. The functor $\sq(\cA|\cB) \to \sD(\cA|\cB)$ induced by Theorem~\ref{theorem. lyubashenko-manzyuk universal property} admits an inverse up to natural equivalence of functors.
\end{theorem}

\begin{theorem}\label{theorem. A oo quotient}
When $\cA$ is cofibrant and $\cB \subset \cA$ is a full subcategory, the maps $\cA \to \sD(\cA|\cB)$ and  $\cA \to \sq(\cA|\cB)$ exhibit $\sD(\cA|\cB)$ and $\sq(\cA|\cB)$ as quotients of $\cA$ along $\cB$ in the sense of Definition~\ref{defn. Aoo quotient}.
\end{theorem}

\begin{proof}
By Theorem~\ref{theorem. lyubashenko-manzyuk different quotients agree}, we need only prove the result for the map $\cA \to \sq(\cA|\cB)$.
We have an induced map
	\eqnn
	\hom_{\Ainftycat}(\sq(\cA|\cB), - ) \to \hom_{\Ainftycat}(\cA, -)
	\eqnd
which, on each test object $\cD$, is induced by the restriction functor from Theorem~\ref{theorem. lyubashenko-manzyuk universal property}. By (taking the nerve and underlying Kan complexes of the functors $A_\infty$-categories in) that theorem, this identifies $\hom_{\Ainftycat}(\sq(\cA|\cB),\cD)$ with the space of those functors $\cA \to \cD$ that are contractible along $\cB$. By Proposition~\ref{prop. pushout diagram is automatic for quotients}, this proves that the map $\cA \to \sq(\cA|\cB)$ exhibits $\sq(\cA|\cB)$ as a quotient.
\end{proof}

\section{Localizations of $A_\infty$-categories}\label{section. Aoo-localization}

In~\cite{gps-covariant}, given an $A_\infty$-category $\cA$ with morphisms satisfying a cofibrancy condition, and given a collection of morphisms $W \subset \cA$, a new $A_\infty$-category is constructed which we will denote by $\cA[W^{-1}]$. This $A_\infty$-category is called a localization in ibid. The goal of this section is to prove the universal property of $\cA[W^{-1}]$ to justify this nomenclature.

\subsection{Localizations of $A_\infty$-categories}

\begin{prop}
\label{prop. three localizations}
Let $\cA$ and $\cL$ be cofibrant $A_\infty$-categories, and let $\iota: \cA \to \cL$ be a functor. We fix also a collection of connected components $W \subset \hom_{\sset}(\Delta^1,N\cA)$. (Put another way, $W$ is a subset of morphisms in $H^0 \cA$.)

The following are equivalent:
\enum[(a)]
	\item\label{item. localization hom space} For every $A_\infty$-category $\cE$, the pullback map of spaces
		\eqnn
		\iota^*: \hom_{\Ainftycat}(\cL,\cE)
		\to \hom_{\Ainftycat}(\cA,\cE)
		\eqnnd
	is fully faithful (meaning it is an injection on $\pi_0$ and a homotopy equivalence along connected components) and the image in $\pi_0$ consists of those functors $f: \cA \to \cE$ for which  $\alpha \in W$ implies $f(\alpha)$ is an equivalence in $\cE$.
	\item\label{item. localization hom cat} For every $A_\infty$-category $\cE$, the pullback map of $\infty$-categories
		\eqnn
		\iota^*:  N\underline{\hom}_{\Ainftycat}(\cL,\cE)
		\to N\underline{\hom}_{\Ainftycat}(\cA,\cE)
		\eqnnd
	is fully faithful and the essential image consists of those functors $f: \cA \to \cE$ for which  $\alpha \in W$ implies $f(\alpha)$ is an equivalence in $\cE$.
	\item\label{item. localization hom enriched} For every $A_\infty$-category $\cE$, the pullback map of $A_\infty$-categories
		\eqn\label{eqn. l star for enriched}
		\iota^*: \underline{\hom}_{\Ainftycat}(\cL,\cE)
		\to \underline{\hom}_{\Ainftycat}(\cA,\cE)
		\eqnd
	is fully faithful and the essential image consists of those functors $f: \cA \to \cE$ for which  $\alpha \in W$ implies $f(\alpha)$ is an equivalence in $\cE$.
\enumd
\end{prop}.  

\begin{proof}

The implications
$\eqref{item. localization hom enriched} \implies
\eqref{item. localization hom cat} \implies
\eqref{item. localization hom space}$ 
are obvious, as both $N$ and the underlying $\infty$-groupoid functors preserve fully faithful maps.

$
\eqref{item. localization hom space}
\implies
\eqref{item. localization hom cat}
$.
The hypothesis is true for all $\cE$, and in particular for $\fun_{A_\infty}(\kk[\Delta^n],\cE)$ where $\kk[\Delta^n]$ is the free $\kk$-linear dg category generated by the $n$-simplex. We have
	\begin{align}
	\hom_{\Ainftycat}(\cA,\fun_{A_\infty}(\kk[\Delta^n],\cE))
		&\simeq \hom_{\Ainftycat}(\cA \tensor_{\Ainftycat} \kk[\Delta^n], \cE) \nonumber \\
		&\simeq \hom_{\Ainftycat}(\kk[\Delta^n], \fun_{A_\infty}(\cA,\cE) )\nonumber \\
		&\simeq \hom_{\inftyCat}(\Delta^n, N\fun_{A_\infty}(\cA,\cE)) \nonumber 
	\end{align}
where the first two equivalences are by using that $\Ainftycat$ is closed symmetric monoidal, and the last equivalence follows by adjunction. (The nerve is the right adjoint to the free $A_\infty$-category functor by Proposition~3.39 of~\cite{tanaka-Aoo-units}.) Because the above equivalences are natural in $n$, we conclude that there are fully faithful inclusions (i.e., inclusions of connected components)
	\eqnn
	\hom_{\inftyCat}(\Delta^n, N\fun_{A_\infty}(\cL,\cE))
	\to
	\hom_{\inftyCat}(\Delta^n, N\fun_{A_\infty}(\cA,\cE))
	\eqnnd
compatible with the simplicial maps in the $n$ variable. Thus, we have a levelwise inclusion of connected components in the map of simplicial spaces
	\eqnn
	\hom_{\inftyCat}(\Delta^\bullet, N\fun_{A_\infty}(\cL,\cE))
	\to
	\hom_{\inftyCat}(\Delta^\bullet, N\fun_{A_\infty}(\cA,\cE))
	\eqnnd
meaning this is then a fully faithful functor of $\infty$-categories, by considering these simplicial spaces as complete Segal spaces. 

$
\eqref{item. localization hom cat}
\implies
\eqref{item. localization hom enriched}
$. Let us first assume that $\cE$ is stable. Then the internal hom $\underline{\hom}_{\Ainftycat}(A,\cE)$ is stable. In particular, for every functor $h: A \to \cE$, and for every integer $k$, we know there exist natural isomorphisms
	\eqn\label{eqn. shift cohomology localization}
	H^k\hom_{\underline{\hom}_{\Ainftycat}(\cA,\cE)}(-,h)
	\cong
	H^{0}\hom_{\underline{\hom}_{\Ainftycat}(\cA,\cE)}(-,\Sigma^{k} h) 
	\eqnd
where $\Sigma^k$ is the $k$-fold shift. To show~\eqref{eqn. l star for enriched} is fully faithful, fix two functors $f,g : \cL \to \cE$. We must show that for every $k$,
	\eqnn
	H^k\hom_{\underline{\hom}_{\Ainftycat}(\cL,\cE)}(f,g)
	\cong
	H^k\hom_{\underline{\hom}_{\Ainftycat}(\cA,\cE)}(f\iota ,g\iota )
	\eqnnd
For non-positive values of $k$, the above isomorphism follows by hypothesis 
\eqref{item. localization hom cat}. On the other hand, for positive $k$, we have a commuting diagram
	\eqnn
	\xymatrix{
	H^k\hom_{\underline{\hom}_{\Ainftycat}(\cL,\cE)}(f,g)
		\ar[d]^{\iota^*}
		\ar[r]_-{\cong}^-{\eqref{eqn. shift cohomology localization}}
		&	
		H^0\hom_{\underline{\hom}_{\Ainftycat}(\cL,\cE)}(f,\Sigma^k g) \ar[d]^{\iota^*}_{\cong}
		\\
	H^k\hom_{\underline{\hom}_{\Ainftycat}(\cA,\cE)}(f\iota ,g\iota )
		\ar[r]_-{\cong}^-{\eqref{eqn. shift cohomology localization}}
		&
		H^0\hom_{\underline{\hom}_{\Ainftycat}(\cA,\cE)}(f\iota ,\Sigma^k (g\iota ) )	.
	}
	\eqnnd
The righthand vertical arrow is an isomorphism by hypothesis -- note we use that $\iota$ is exact (Remark~\ref{remark. every functor is exact}) to naturally identify $(\Sigma^k g )\iota$ and $\Sigma^k(g\iota )$. It follows that $\iota^*$ is fully faithful.

When $\cE$ is not stable, consider the fully faithful embedding to the triangulated closure $j: \cE \to \Tw \cE$ (Proposition~\ref{prop. Tw facts}\eqref{item. A to Tw A}). For any domain $A_\infty$-category, post-composition by $j$ also realizes a triangulated closure (Proposition~\ref{prop. Tw facts}\eqref{item. Tw of Fun is functors into Tw}):
	\eqnn
	\underline{\hom}_{\Ainftycat}(-,\cE)
	\to
	\underline{\hom}_{\Ainftycat}(-,\Tw \cE)
	\simeq
	\Tw \underline{\hom}_{\Ainftycat}(-,\cE).
	\eqnnd
Thus the above composition is fully faithful, and we can reduce to the case of stable $\cE$.	
\end{proof}

\begin{defn}
\label{defn. localization of Aoo cat}
Let $\cA$ be an $A_\infty$-category and fix $W \subset \hom_{\sset}(\Delta^1,N\cA)$ a collection of connected components of morphisms in $\cA$.

We say a functor of $A_\infty$-categories $\iota: \cA \to \cL$ is a {\em localization} of $\cA$ along $W$ if for any (and hence for all) commutative diagrams $\Delta^1 \times \Delta^1 \to \Ainftycat$ 
	\eqn\label{eqn. localization definition}
	\xymatrix{
	\cA' \ar[r]^{\iota'} \ar[d]
	& \cL ' \ar[d] \\
	\cA \ar[r]^{\iota}
	& \cL
	}
	\eqnd
where the vertical maps are isomorphisms in $\Ainftycat$ and $\cA',\cL'$ are cofibrant,
the map $\iota'$ satisfies any of the equivalent properties of Proposition~\ref{prop. three localizations}.
\end{defn}

\begin{remark}
Let $\cA$ be a cofibrant $A_\infty$-category, and let $W \subset H^0\hom_{\cA}$ be a collection of (cohomology classes of) morphisms in $\cA$. Suppose further that $\cL$ is cofibrant. Then $\iota: \cA \to \cL$ is a localization if and only if, for every $A_\infty$-category $\cE$, the map
	\eqnn
	\iota^*: \fun_{A_\infty}(\cL,\cE)
	\to
	\fun_{A_\infty}(\cA,\cE)
	\eqnnd
is fully faithful. One can, for example, take $\iota' = \iota$ and the vertical arrows to be the identity maps in~\eqref{eqn. localization definition}.
\end{remark}

\begin{remark}
The universal property of $\cA[W^{-1}]$ guarantees the following: If $f: \cA \to \cD$ is any functor of $A_\infty$-categories for which every $w \in W$ is sent to an equivalence in $\cD$, then $f$ factors through $\cA[W^{-1}]$ up to homotopy of functors. Let us explain how Definition~\ref{defn. localization of Aoo cat} guarantees this, for those who are not used to such definitions. The map denoted $\iota^*$ in that definition is, up to homotopy equivalence, just an inclusion of connected components---hitting exactly those connected components of $\hom_{\Ainftycat}(\cA,\cD)$ containing functors that send $W$ to equivalences. Thus, by writing down what it means to choose a homotopy inverse to $\iota^*$ from the image of $\iota^*$, we are supplied an extension $\tilde f: \cA[W^{-1}] \to \cD$, along with a homotopy from $\tilde f \circ \iota$ to $f$. This data is canonical, in that the space of homotopy inverses (together with homotopies exhibiting the homotopy inverse as a homotopy inverse) is a contractible space.
\end{remark}

\subsection{Localizations via quotients}

Now let us construct localizations from the existence of quotients.

\begin{remark}
Note that, in the setting of dg-categories, To\"en proceeds in the reverse direction;  To\"en constructs localizations, then concludes the existence of quotients~\cite{toen-homotopy-theory-of-dg-cats}.
\end{remark}

\begin{notation}[$\cB_W$]\label{notation. cB_W}
Let $\cA$ be an $\kk$-linear $A_\infty$-category and fix a collection $W \subset H^0(\hom_\cA)$. We let $\cB_W \subset \Tw\cA$ denote the full subcategory of those objects arising as mapping cones of elements of $W$.
\end{notation}

\begin{prop}
\label{prop. Bw contractible is localizing}
Let $F: \Tw \cA \to \cD$ be a functor. $F$ is contractible along $\cB_W$ if and only if it sends morphisms of $W$ to equivalences in $\cD$. 
\end{prop}

\begin{proof}
First recall that $F$ is automatically an exact functor by Remark~\ref{remark. every functor is exact}.
 
For one direction: Given $w: y \to z$ a closed degree 0 morphism (whose cohomology class is) in $W$, and for any object $x \in \ob \cD$, we have that
	\eqnn
	\cone(\hom_{\cD}(x,Fy) \xrightarrow{(Fw)_\ast}\hom_{\cD}(x,Fz))
	\simeq
	\hom_{\cD}(x,F\cone(w)).
	\eqnnd
We have used that $F$ is exact, and that $\hom(-,\cone(y \to z))$ is, as a chain complex, a mapping cone of $\hom(-,y) \to \hom(-,z)$ (see Example~2.7 of~\cite{tanaka-Aoo-units}). By the assumption that $F$ is contractible along $\cB_W$, we know that the identity chain map of the complex $\hom_{\cD}(F\cone(w),F\cone(w))$ is null-homotopic. In particular -- by post-composing with a unit of $F\cone(w)$, say -- $\hom_{\cD}(x,F\cone(w))$ is chain homotopy equivalent to zero. But (by standard homological algebra) a contracting homotopy for a mapping cone of a chain map exhibits an inverse-up-to-homotopy of the chain map. In particular, for all $x \in \ob \cD$, the chain map
	\eqnn
	\hom_{\cD}(x,Fy) \xrightarrow{(Fw)_\ast}\hom_{\cD}(x,Fz)
	\eqnnd
admits an inverse-up-to-homotopy chain map $\phi_x$. Plugging in $x = Fz$, and choosing a unit for $Fz$, one obtains a degree 0 element $\phi_{Fz}(e_{Fz})\in \hom_{\cD}(Fz,Fy)$. One easily checks this exhibits an equivalence $Fz \to Fy$ in $\cD$, with inverse up to homotopy given by $(Fw)_\ast e_{Fy}$ (this last notation takes $x = Fy$).  But $(Fw)_\ast e_{Fy}$ is $m^2(Fw, e_{Fy})$, which by definition of unit is homotopic to $Fw$. Thus, $Fw$ is an equivalence in $\cD$.

For the reverse direction: 
If $Fw$ admits an inverse up to homotopy in $\cD$, one obtains an inverse up to chain homotopy of the chain map
	\eqnn
	\hom_{\cD}(\cone(Fw),Fy) \xrightarrow{(Fw)_\ast}\hom_{\cD}(\cone(Fw),Fz)
	\eqnnd
and, as we have already mentioned, the mapping cone of this chain map is chain homotopy equivalent to $\hom_{\cD}(\cone(Fw),\cone(Fw)).$ Thus (again by homological algebra) the inverse-up-to-homotopy of the above chain map exhibits a contraction of the chain complex $\hom_{\cD}(\cone(Fw),\cone(Fw)).$ In particular, any unit $e$ for $\cone(Fw)$ is cohomologous to zero. We conclude that any two objects $x,y$ in $\cB_W$, the map $\hom_{\Tw \cA}(x,y) \to \hom_{\Tw \cD}(Fx,Fy)$ is null homotopic (say, by pre- or post-composing by a unit).
\end{proof}

\begin{definition}[$\cL_W$]\label{defn. Lw}
Now consider the quotient  $(\Tw \cA) / \cB_W$. The quotient receives a natural functor from $\cA$, given by the composite $\cA \to \Tw \cA \to (\Tw \cA) / \cB_W$.

We let $\cL_W \subset (\Tw \cA) / \cB_W$ be the full subcategory consisting of objects in the essential image of $\cA \to(\Tw \cA) / \cB_W$.
\end{definition}

Our goal is to prove the following:

\begin{theorem}\label{thm. localization of A oo cats}
When $\cA$ is cofibrant, the functor $\cA \to \cL_W$ exhibits $\cL_W$ as a localization of $\cA$ along $W$.
\end{theorem}

We first begin with a stable analogue:

\begin{lemma}\label{lemma. localization for stable things}
Let $\cD$ be stable.
Assume $\cA$ (and hence $\Tw \cA$) is cofibrant. Then restriction along $\iota: \Tw \cA \to \cL_W$ induces an inclusion of connected components
	\eqnn
	\hom_{\Ainftycat}((\Tw \cA) / \cB_W, \cD)
	\to
	\hom_{\Ainftycat}(\Tw \cA, \cD).
	\eqnd
where the essential image of this inclusion consists of those functors sending every morphism in $W$ to an equivalence.
\end{lemma}

\begin{proof}
By the universal property of quotients, we know $\iota^*$ is fully faithful (i.e., an inclusion on connected components). We further know that the restriction along $\iota$ has essential image given by those functors $\Tw \cA \to \cD$ that are contractible along $\cB_W$. 
Because $\cD$ is stable, by Proposition~\ref{prop. Bw contractible is localizing}, the essential image thus consists of those functors sending elements of $W$ to invertible-up-to-homotopy morphisms in $\cD$. 
\end{proof}

\begin{lemma}\label{lemma. Tw Lw}
The map in $\Ainftycatt$
	\eqnn
	\Tw \cL_W \to \Tw \left((\Tw \cA) / \cB_W\right)
	\eqnd
induced by $\cL_W \to (\Tw \cA) / \cB_W$ (Definition~\ref{defn. Lw}) descends to an equivalence in $\Ainftycat$.
\end{lemma}

\begin{proof}
We have a strictly commutating diagram in $\Ainftycatt$
	\eqnn
	\xymatrix{
	\cA \ar[r] \ar[drr] & \Tw \cA \ar[r] \ar@{-->}[drr] & (\Tw \cA)/\cB_W  \ar[r] &  \Tw \left((\Tw \cA) / \cB_W\right) \\
	&&\cL_W \ar[u] \ar[r] & \Tw \cL_W \ar[u]_j 
	}.
	\eqnd
The arrow $\cA \to \cL_W$ is  essentially surjective by definition of $\cL_W$, and extends to the dashed arrow by naturality (Proposition~\ref{prop. Tw facts}\eqref{item. A to Tw A}).

First, let us see that the map in question (which we have labeled as $j$ for the purposes of this proof)  is fully faithful. To see this, write an object $X$ of $\Tw\cL_W$ as a finite iterated extension (possibly with shifts) of objects $X'_i, i = 1, \ldots, m$ of $\cL_W$. Given any other object $Y$, write it as an extension  (possibly with shifts) of objects $Y'_k \in \cL_W, k = 1, \ldots, n$. Then, by the exactness of $j$ (Remark~\ref{remark. every functor is exact}), $j(X)$ and $j(Y)$ are extensions of $j(X_i)$ and $j(Y_k)$, respectively. Because the map $\cL_W \to \Tw\left( (\Tw\cA)/\cB_W\right)$ is fully faithful, the maps 
	\eqn\label{eqn. associated graded hom maps}
	\hom_{\Tw\cL_W}(X_i,Y_k) \to \hom_{\Tw\left( (\Tw\cA)/\cB_W\right)}(j(X_i),j(Y_k))
	\eqnd
	are quasi-isomorphisms. On the other hand, the maps 
	\eqnn
	H^*\hom_{\Tw\cL_W}(X,Y) \to H^* \hom_{\Tw\left( (\Tw\cA)/\cB_W\right)}(j(X),j(Y))
	\eqnd
fit into an iteration of maps of long exact sequences (because mapping cone sequences induces long exact sequences of cohomology groups). Four-fifths of these maps are isomorphisms of the form \eqref{eqn. associated graded hom maps}. By the Five Lemma and induction on $m,n$, fully faithfulness of $j$ follows.

Let us show that $j$ is essentially surjective. By the definition of $\Tw$, any object $Y \in \Tw\left( (\Tw\cA)/\cB_W\right)$ can be expressed as an iterated mapping cone of morphisms (and their shifts) $f_i: X_i \to X_{i+1}$ in $(\Tw \cA)/\cB_W$. Thus it suffices to show that any object $X \in (\Tw \cA)/\cB$ is in the essential image of $j$. (Then, because $j$ is fully faithful, it follows that the morphisms $f_i$ are also in the image of $j$; because $j$ is exact and $\Tw \cL_W$ is stable, we conclude $Y$ is in the essential image of $j$.)

Because the quotient map $\Tw \cA \to (\Tw \cA)/\cB_W$ is essentially surjective, each object $X \in (\Tw\cA)/\cB_W$ arises from some object $\tilde X \in \Tw \cA$. In turn, $\tilde X$ is an iterated mapping cone of morphisms $\tilde h_k: \tilde W_{k} \to \tilde W_{k+1} \in \cA$ and their shifts (again by definition of $\Tw$).

We let $h_k$ be the images of $\tilde h_k$ in $\cL_W$ (and in $\Tw \cL_W$, by abuse of notation). Because the dashed arrow $\Tw \cA \to \Tw \cL_W$ is exact, we conclude that the iterated cones of the morphisms $h_j$ and their shifts are equivalent to $X$. This establishes the claim that $j$ is essentially surjective.

Because $j$ is essentially surjective and fully faithful, we are finished.
\end{proof}

\begin{proof}[Proof of Theorem~\ref{thm. localization of A oo cats}.]
Because $\cA$ and hence $\cL_W$ are cofibrant, we may interpret $\hom_{\Ainftycat}$ as a space of functors, and we identify $\hom_{\Ainftycat}(\cA,\cD)$ as the full subspace of $\hom_{\Ainftycat}(\Tw\cA, \Tw \cD)$ spanned by those $f$ that factor $\cA$ through the essential image of $\cD \subset \Tw \cD$.

Let us now consider the iterated pullback squares of $\infty$-groupoids
	\eqnn
	\xymatrix{
	Q \ar[d]	\ar[r]		& \hom_{\Ainftycat}(\cA,\cD) \ar[d] \\
	Q' \ar[d]^{\simeq} \ar[r]  & \hom_{\Ainftycat}(\cA, \Tw \cD) \ar[d]^{\simeq}\\
	\hom_{\Ainftycat}((\Tw \cA)/ \cB_W, \Tw \cD) \ar[r]		& \hom_{\Ainftycat}(\Tw \cA, \Tw \cD).
	}
	\eqnd
(We remind the reader that in the $\infty$-category of $\infty$-groupoids, the only notion of pullback that makes sense is what one might classically call a homotopy pullback; in particular, the above are all what one might call homotopy pullback squares.)
The lower-right vertical arrow is a homotopy equivalence by adjunction; the lower-left vertical arrow is a homotopy equivalence being the pullback of a homotopy equivalence. Note that the top-right vertical arrow is fully faithful (that is, an injection on $\pi_0$ and a homotopy equivalence along each connected component). Because the bottom horizontal arrow is also fully faithful, it follows that  the top-left vertical arrow is fully faithful.

By considering the outermost rectangle, we see that $Q$ is identified as the space of functors $\cA \to \cD$ sending morphisms in $W$ to equivalences (by Lemma~\ref{lemma. localization for stable things}).

Now we examine the two inner pullback squares. First, note that one can identify $Q'$ with $\hom_{\Ainftycat}( \cL_W, \Tw \cD)$ by the definition of $\cL_W$. Here are the details: We have the composition of maps between Kan complexes
 	\begin{align}
	\hom_{\Ainftycat}( (\Tw \cA) / \cB_W, \Tw \cD)
	&\to
	\hom_{\Ainftycat}( \Tw \left((\Tw \cA) / \cB_W\right), \Tw \cD) \nonumber \\
	&\to 
	\hom_{\Ainftycat}( \Tw \cL_W, \Tw \cD) \nonumber
	\\ & \to
	\hom_{\Ainftycat}( \cL_W, \Tw \cD). \nonumber
	\end{align}
The first arrow is a homotopy equivalence by the universal property of $\Tw$ (i.e., adjunction---see Remark~\ref{remark. universal property of Tw}).
The second arrow is a homotopy equivalence by Lemma~\ref{lemma. Tw Lw}, and the third is a homotopy equivalence by the universal property of $\Tw$ yet again.

Tracing through these equivalences, we see that the horizontal arrow from $Q'$ can further be identified with restriction along $\cA \to \cL_W$. Because $\cA \to \cL_W$ is essentially surjective, this gives a second identification of $Q$: The space of functors from $\cL_W$ to $\Tw \cD$ essentially factoring through $\cD$; that is, functors from $\cL_W$ to $\cD$. 
\end{proof}

\subsection{Naturality of localizations}

\begin{proposition}\label{prop. naturality of localizations}
Let $\cC$ be an ordinary category, and fix a functor $F: \cC \to \Ainftycatt$. Moreover suppose that for every $x \in \cC$, we have chosen a collection $W_x \subset H^0\hom_{F(x)}$, and that for every morphism $x \to y$ in $\cC$, we have that $W_x$ has image contained in $W_y$.

Then localization induces a functor of $\infty$-categories
	\eqnn
	N(\cC) \to \Ainftycat
	\eqnd
sending $x$ to $F(x)[W_x^{-1}]$, and sending a morphism $x \to y$ to the induced functor $F(x)[W_x^{-1}] \to F(y)[W_y^{-1}]$.
\end{proposition}

\begin{proof}
$F$ defines a diagram $\cC \times \Delta^1 \to \Ainftycat$ whose value at $\{x\} \times \Delta^1$ is given by
	\eqnn
	\cB_{W_x} \to \Tw F(x).
	\eqnd
By the naturality of colimits (in particular, quotients) and functoriality of $\Tw$, we then have an induced diagram $\cC \times \Delta^1 \to \Ainftycat$ whose value on $\{x\} \times \Delta^1$ is
	\eqnn
	a_x: F(x) \to (\Tw F(x)) / \cB_{W_x}.
	\eqnd
We have the map $N(\cC) \to \Ainftycat$ by taking induced arrows among the essential images of $a_x$.
\end{proof}

\subsection{Model for localizations}

\begin{theorem}\label{thm. LM model of localization}
Let $\cA$ be an arbitrary $\kk$-linear $A_\infty$-category and $W \subset H^0\hom_{\cA}$ a collection of (cohomology classes of) morphisms. Then the localization $A[W^{-1}]$ (Definition~\ref{defn. localization of Aoo cat}) is equivalent to the full subcategory of $\sD( \Tw \cA | \cB_W)$ spanned by the essential image of the composite $\cA \to \Tw \cA \to \sD(\Tw \cA | \cB_W)$.
\end{theorem}

\begin{proof}
Combine Theorem~\ref{theorem. A oo quotient} with Theorem~\ref{thm. localization of A oo cats}.
\end{proof}

\subsection{Morphisms in a localization}

Fix an $A_\infty$-category $\cA$ and a collection of (cohomology classes of) morphisms $W \subset H^0\hom_{\cA}$.
Fix two objects $X, Y \in \cA$. By Corollary~\ref{thm. LM model of localization}, the morphism complex $\hom_{\cA[W^{-1}]}(X,Y)$ in the localized category may be computed as 
	\eqnn
	\hom_{\sD(\Tw \cA | \cB_W)}(X,Y).
	\eqnd
By definition~\cite{lyubashenko-manzyuk, gps-covariant}, this complex is given by a bar-type construction:
	\begin{gather}
	\vdots \nonumber  \\
	\oplus_{Z_1, Z_2 \in \cB_W} \hom_{\Tw\cA}(X,Z_1) \tensor  \hom_{\Tw\cA}(Z_1,Z_2) \tensor  \hom_{\Tw\cA}(Z_2, Y) \nonumber \\
		\downarrow \nonumber \\
	\oplus_{Z_1 \in \cB_W} \hom_{\Tw\cA}(X,Z_1) \tensor  \hom_{\Tw\cA}(Z_1, Y) \nonumber \\
		\downarrow \nonumber \\
	\hom_{\cA}(X,Y) \label{eqn. localization hom as bar complex}
	\end{gather}
We use the following lemma to compute morphism complexes in our wrapped Fukaya categories in \cite{oh-tanaka-actions}. A dual assertion is made in Lemmas~3.11 and 3.14 of~\cite{gps-covariant}.

\begin{lemma}\label{lemma. local sequences localize}
Fix cofibrant $A_\infty$-category $\cA$.
Fix also a sequence of objects $Y_0 \to Y_1 \to \ldots$ in $\cA$. (That is, a collection of objects $Y_i$ equipped with morphisms $Y_{i} \to Y_{i+1}$.) Suppose moreover that for any morphism $Q \to Q'$ in $W$, the induced map
	\eqn\label{eqn. lemma localize sequence hypothesis}
	\hocolim_i \hom_{\cA} (Q', Y_i) \to
	\hocolim_i \hom_{\cA} (Q, Y_i)
	\eqnd
is a quasi-isomorphism. Then for any $X \in \cA$, the induced map
	\eqnn
	\hocolim_i \hom_{\cA} (X,Y_i)
	\to
	\hocolim_i \hom_{\cA[W^{-1}]}(X,Y_i)
	\eqnd
is a quasi-isomorphism.
\end{lemma}

\begin{remark}
For the sake of coherence (and to define the homotopy colimit), one should extend the collection of morphisms $Y_i \to Y_{i+1}$ to a functor $\ZZ_{\geq 0} \to N(\cA)$, where $N(\cA)$ is the $A_\infty$-nerve and $\ZZ_{\geq 0}$ is the partially ordered set of non-negative integers. However, there is a contractible choice of such extensions, so we ignore this detail.
\end{remark}

\begin{proof}
A sequential homotopy colimit may be computed using a mapping telescope construction. The telescope construction commutes with direct sums and tensor products, so for any $i$,
	\eqnn
	\hocolim_i \hom_{\cA[W^{-1}]}(X,Y_i)
	\eqnd
may be computed by a chain complex built from a simplicial object (just as \eqref{eqn. localization hom as bar complex} is) where the $k$-simplices are as follows:
	\begin{align}
	\bigoplus_{Z_1,\ldots,Z_k}
	&
	\hom_{\Tw\cA}(X,Z_1)
	\tensor \hom_{\Tw \cA}(Z_1,Z_2)
	\tensor \ldots 
	\nonumber \\
	&\tensor \hom_{\Tw \cA}(Z_{k-1},Z_k)
	\tensor
	\left( \hocolim_i  \hom_{\Tw\cA}(Z_k,Y_i)
	\right)\nonumber 
	\end{align}
We now claim that for every $k$, $\hocolim_i  \hom_{\Tw\cA}(Z_k,Y_i)$ is acyclic because $Z_k$ arises as a cone of a morphism $Q \to Q'$ in $W$. That is, the mapping cone sequence $Q \to Q' \to Z_k$ induces a mapping cone sequence
	\eqnn
	\hom_{\Tw\cA}(Z_k,Y_i) \to
	\hom_{\Tw\cA}(Q',Y_i) \to
	\hom_{\Tw\cA}(Q,Y_i)
	\eqnd
of cochain complexes; since the mapping telescope construction commutes with mapping cones, the diagram
	\eqnn
	\xymatrix{
	\hocolim_i \hom_{\Tw\cA}(Z_k,Y_i) \ar[r] \ar[d]
		& \hocolim_i \hom_{\Tw\cA}(Q',Y_i) \ar[d] \\
	0 \ar[r]
		& \hocolim_i \hom_{\Tw\cA}(Q,Y_i)
	}
	\eqnd
is still a homotopy pushout; in particular, the right vertical map being a quasi-isomorphism by assumption~\eqref{eqn. lemma localize sequence hypothesis}, the top-left cochain complex is acyclic.

Finally, we note that if each morphism complex in $\cA$ is homotopically flat (i.e., $\cA$ is cofibrant), the same holds for $\Tw \cA$ (as morphism complexes are defined as iterated mapping cones and shifts of the morphism complexes of $\cA$). This means that all the terms in the bar complex are acyclic except for the 0-simplex term, which is $\hocolim_i \hom_{\cA}(X,Y_i)$ to begin with. Now a standard argument, for example using the length filtration and seeing that the associated graded are all acyclic, gives the statement we desire.
\end{proof}

\bibliographystyle{plain}
\bibliography{20240601-biblio}

\end{document}